\documentclass[pdflatex,11pt]{article}
\usepackage[T1]{fontenc}
\usepackage[applemac]{inputenc}
\usepackage[english]{babel}
\usepackage{geometry}
\usepackage{amsfonts,amsmath,amssymb,amsthm}
\usepackage{lmodern}
\usepackage{graphicx}
\usepackage{enumerate}
\usepackage{hyperref}

\makeatletter
\newtheorem*{rep@theorem}{\rep@title}
\newcommand{\newreptheorem}[2]{%
\newenvironment{rep#1}[1]{%
 \def\rep@title{#2 \ref{##1}}%
 \begin{rep@theorem}}%
 {\end{rep@theorem}}}
\makeatother

\newtheorem{theo}{Theorem}[section]
\newtheorem*{theo*}{Theorem}
\newreptheorem{theo}{Theorem}
\newtheorem{prop}[theo]{Proposition}
\newtheorem{lem}[theo]{Lemma}
\newtheorem{cor}[theo]{Corollary}
\newtheorem*{conj}{Conjecture}

\theoremstyle{definition}
\newtheorem{defi}[theo]{Definition}

\newtheorem{notation}[theo]{Notation}
\newtheorem*{ack}{Acknowledgments}

\theoremstyle{remark}
\newtheorem{rem}[theo]{Remark}
\newtheorem*{rem*}{Remark}

\newcommand\N{\mathbb N}
\newcommand\C{\mathbb C}
\newcommand\B{\mathcal B}

\newcommand\graphe{G_{lw}}

\newcommand\graphek{\graphe^{(k)}}
\newcommand\graphekw{\tilde G_{lw}^{(k)}}
\newcommand\g{\lambda}
\newcommand\gc[1]{\mu_{(#1)}}
\newcommand\gw{\gc{w}}

\newcommand\spa{\gamma_1}
\newcommand\spb{\gamma_2}
\newcommand\gp{\gc{\spa}} 
\newcommand\gpp{\gc{\spb}} 
\newcommand\nb{N}
\newcommand\nbw{\nb^{(w)}}
\newcommand\nbb{N_\circ}
\newcommand\nbbw{\nbb^{(w)}}
\DeclareMathOperator\tr{tr}
\newcommand\disque[1]{D_{#1}}
\DeclareMathOperator\mcg{Mod}

\newcommand\init[1]{S(#1)}
\newcommand\final[1]{F(#1)}

\newcommand\card[1]{\#(#1)}

\author{Sandrine Caruso}
\title{On the genericity of pseudo-Anosov braids I: rigid braids}
\date{}

\begin{document}

\maketitle


\begin{abstract}
We prove that, in the $l$-ball of the Cayley graph of the braid group with $n \geqslant 3$ strands, the proportion of rigid pseudo-Anosov braids is bounded below independently of $l$ by a positive value. 
\end{abstract}

\section{Introduction}

A natural question concerning the Nielsen-Thurston classification of braids is: what is the most likely Nielsen-Thurston type of a ``long random braid''? Different interpretations can be given to this question, but in this paper we shall use the following setting. We consider the Cayley graph of the braid group~$B_n$ (for a fixed number of strands $n$), with generators the set of simple braids -- this is the standard generating set when the braid group is studied as a Garside group. A well-known conjecture since the work of Thurston is as follows.
\begin{conj}
The proportion of pseudo-Anosov braids among all elements in the ball of radius~$l$ in the Cayley graph converges to $1$ as $l$ tends to infinity.
\end{conj}
The best known results going into this direction were, to the best of our knowledge, the classical paper of Fathi~\cite{Fathi}, and the article of Atalan and Korkmaz~\cite{AK} which deals with the case of three-strand braids. The present paper, 
together with the article in preparation \cite{CW}, contains a proof of the above conjecture. In this first part, we introduce some essential tools needed for the proof in \cite{CW}, and already prove a result of independent interest concerning the proportion of 
\emph{rigid} pseudo-Anosov braids (see Corollary~\ref{cor:boule}):
\begin{theo*}
For sufficiently large $l$, the proportion of rigid pseudo-Anosov braids in the ball of radius $l$ in the Cayley graph of~$B_n$ is bounded below by a strictly positive constant which does not depend on~$l$ (but might depend on~$n$).
\end{theo*}
The proof is in two steps: we shall see that the proportion of so-called \emph{rigid} braids is bounded below independently of~$l$, and among rigid braids the proportion of pseudo-Anosov elements converges to~$1$.

Another possible interpretation of the original question should be pointed out: the work of Maher~\cite{maher} and Sisto \cite{sisto} deals with braids obtained by a long random walk in the Cayley graph. They prove that in this setting, as well, the probability of obtaining a pseudo-Anosov braid converges to $1$ as the length of the walk tends to infinity.

\begin{ack}
I would like to thank my PhD advisor Bert Wiest for his help and guidance, Xavier Caruso for fruitful discussions, and Juan Gonz\'alez-Meneses for his careful reading and constructive comments.
\end{ack}

\section{Definitions}

Throughout the article, we fix an integer $n \geqslant 3$. All the considered braids will be braids with $n$ strands.


\subsection{Garside structure}

A general introduction to Garside theory can be found in \cite{DDGM}. The reader can also consult \cite{EM}. We shall only recall some facts which are useful for our purposes.

While the group $\B_n$ admits the well-known presentation of groups
\[\B_n = \left< \sigma_1, \ldots, \sigma_{n-1}\ ;\ \sigma_i \sigma_{i+1} \sigma_i = \sigma_{i+1}\sigma_i \sigma_{i+1} \text{ and } \sigma_i \sigma_j = \sigma_j \sigma_i \text{ for } |i-j| \geqslant 2 \right>,\]
the \emph{monoid of positive braids} $\B_n^+$, which is embedded in $\B_n$, is defined by the same presentation, interpreted as a presentation of monoids.

For $i < j \leqslant n$, we denote by $\Delta_{ij}$ the element of $\B_n^+$ defined by
$$\Delta_{ij} = (\sigma_i \cdots \sigma_{j-1})(\sigma_i\cdots\sigma_{j-2}) \cdots (\sigma_i\sigma_{i+1})\sigma_i$$
and we denote by $\Delta = \Delta_{1n} \in \B_n^+$.

The pair $(\B_n^+,\Delta)$ defines what we call a Garside structure on $\B_n$. Without giving the complete definition, here are some properties of such a structure. The group $\B_n$ is endowed with a partial order $\preccurlyeq$ defined by $x \preccurlyeq y \Leftrightarrow x^{-1}y \in \B_n^+$. If $x \preccurlyeq y$, we say that $x$ is a \emph{prefix} of $y$. Any two elements $x,y$ of $\B_n$ have a unique greatest common prefix.

We also define $\succcurlyeq$ by $x \succcurlyeq y \Leftrightarrow x y^{-1} \in \B_n^+$. Note that $x \succcurlyeq y$ is not equivalent to $y \preccurlyeq x$.

The elements of the set $\{x \in \B_n, 1 \preccurlyeq x \preccurlyeq \Delta\}$ are called \emph{simple braids}.

\begin{prop}
The set of simple braids is in bijection with the set $\mathfrak S_n$ of permutations of $n$ elements, \emph{via} the canonical projection from $\B_n$ to $\mathfrak S_n$.
\end{prop}

\begin{defi}[left-weighting]
Let $s_1$, $s_2$ be two simple braids in $\B_n$. We say that $s_1$ and $s_2$ are \emph{left-weighted}, or that the pair $(s_1,s_2)$ is left-weighted, if there does not exist any generator $\sigma_i$ such that $s_1 \sigma_i$ and $\sigma_i^{-1} s_2$ are both still simple.
\end{defi}

\begin{defi}[starting set, finishing set]
Let $s \in \B_n$ be a simple braid. We call \emph{starting set of $s$} the set $\init s = \{i, \sigma_i \preccurlyeq s\}$ and \emph{finishing set of $s$} the set $\final{s} = \{i, s \succcurlyeq \sigma_i \}$.
\end{defi}

\begin{rem}
Two simple braids $s_1$ and $s_2$ are left-weighted if and only if $\init{s_2} \subset \final{s_1}$.
\end{rem}

\begin{rem}
Let $s$ be a simple braid, and $\pi$ be the permutation associated to $s$. Then $i \in \init s$ if and only if $\pi(i) > \pi(i+1)$, and $i \in \final s$ if and only if $\pi^{-1}(i) > \pi^{-1}(i+1)$.
\end{rem}

\begin{prop}
Let $x \in \B_n$. There exists a unique decomposition $x = \Delta^p x_1\cdots x_r$ such that $x_1, \ldots,x_r$ are simple braids, distinct from $\Delta$ and $1$, and such that $x_i$ and $x_{i+1}$ are left-weighted for all $i = 1, \ldots, r-1$.
\end{prop}

\begin{defi}[left normal form]
In the previous proposition, the writing $x = \Delta^p x_1 \cdots x_r$ is called the \emph{left normal form} of $x$, $p$ is called the \emph{infimum} of $x$ and is denoted by $\inf x$, $p+r$ is the \emph{supremum} of $x$ and is denoted by $\sup x$, and $r$ is called the \emph{canonical length} of~$x$.

Furthermore, if $r \geqslant 1$, we denote by $\iota(x) = \Delta^{p} x_1 \Delta^{-p}$ the \emph{initial factor} of $x$ ($\iota(x) = x_1$ if $p$ is even, $\iota(x) = \Delta x_1 \Delta^{-1}$ if $p$ is odd), and $\phi(x) = x_r$ its \emph{final factor}.
\end{defi}

\begin{defi}[rigidity]
A braid $x$ of positive canonical length is said to be \emph{rigid} if $\phi(x)$ and $\iota(x)$ are left-weighted.
\end{defi}

\subsection{Braids and mapping class group of the punctured disk}

\begin{defi}[Mapping class group of the punctured disk]
Let $\disque n$ be the closed unit disk in $\C$, with $n$ punctures regularly spaced on the real axis. The \emph{mapping class group} of $\disque n$, denoted $\mcg(\disque n)$, is the group of homeomorphisms of $\disque n$, modulo the isotopy relation. We also denote $\mcg(\disque n, \partial \disque n)$ the group of homeomorphisms of $\disque n$ fixing pointwise the boundary $\partial \disque n$ of $\disque n$, modulo the isotopy relation.
\end{defi}

The Artin braid group with $n$ strands is isomorphic to the group $\mcg(\disque n, \partial \disque n)$.

Recall that the classification theorem of Nielsen and Thurston states that a mapping class $f \in \mcg(\disque n)$ is exactly one of the following: periodic, or reducible non-periodic, or pseudo-Anosov. A braid $x \in \mcg(\disque n, \partial \disque n)$ can be projected on an element of $\mcg(\disque n)$. We call \emph{Nielsen-Thurston type of $x$} the Nielsen-Thurston type of its projection. The definition of periodicity is then transformed as follows: a braid $x \in \B_n$ is periodic if and only if there exist nonzero integers $m$ and $l$ such that $x^m = \Delta^l$, where $\Delta = (\sigma_1 \cdots \sigma_{n-1})(\sigma_1\cdots\sigma_{n-2}) \cdots (\sigma_1\sigma_2)\sigma_1$. (Geometrically $\Delta$ corresponds to the half-twist around the boundary of the disk).

\subsection{Round curves and almost round curves}

Let us consider a braid as a mapping class in the mapping class group $\mcg(\disque n, \partial \disque n)$.

\begin{defi}[curve]
We call \emph{closed curve} in $\disque n$ the image of the circle $\mathbb S^1$ by a continuous map with values in $\disque n$. The curve is said to be \emph{simple} if this map is injective. It is said to be \emph{non degenerated} if it is neither homotopic to a point, nor to the boundary of the disk, and it bounds a least two punctures.
\end{defi}

In the following, we simply call curve a homotopy class of non degenerate simple closed curves.

\begin{defi}[round curve]
A curve is said to be \emph{round} if it is represented by a circle in $\disque n$.
\end{defi}

\begin{defi}[almost round curve]
A curve is said to be \emph{almost round} if it is not round, and is the image by a simple braid of a round curve.
\end{defi}

\section{Properties of the left-weighting graph}

\begin{defi}[Left-weighting graph]
We call \emph{left-weighting graph}, denote by $\graphe$, the following finite oriented graph. The vertices are indexed by the simple braids except $1$ and $\Delta$, and there is an edge from the vertex $x_1$ to the vertex $x_2$ if and only if the pair $(x_1,x_2)$ is left-weighted.

We call \emph{path} a sequence $(x_1 \to x_2 \to \cdots \to x_l)$ such that there is an edge from the vertex $x_i$ to the vertex $x_{i+1}$, and the \emph{length} of such a path means the number of edges in the path.
\end{defi}

The objective of this section is to study some properties of the graph $\graphe$, especially some asymptotic properties of the number of paths of length $l$, with $l$ tending to infinity. We introduce the following notations, for all $l \in \N^\ast$:
\begin{itemize}
\item $\nb(l)$ is the number of paths $(x_1 \to x_2 \to \cdots \to x_{l+1})$ of length $l$ in $\graphe$;
\item $\nbb(l)$ is the number of loops of length $l+1$, with marked base vertex, in $\graphe$. The quantity $\nbb(l)$ can also be seen as the number of paths of length $l$, such that there is an edge from the last to the first vertex.
\item Let $w$ be a path of length $k \in \N^\ast$ in $\graphe$. We denote by $\nbw(l)$ the number of paths of length $l$ in $\graphe$ that do not pass through $w$ (\emph{ie} that do not contain $w$ as a subpath), and $\nbbw(l)$ the number of loops of length $l+1$ with marked base vertex in $\graphe$, that do not pass through $w$.
\end{itemize}

\bigskip

Furthermore, if $(u_l)$ and $(v_l)$ are two sequences of real numbers, we write $u_l = \Theta(v_l)$ if and only if there exist constants $c_1, c_2 > 0$ such that for all large enough $l$, $c_1 v_l < u_l < c_2 v_l$. We say that \emph{$u_l$ is of the order of $v_l$}.

We also use the usual notations $u_l \sim v_l$ when \emph{$u_l$ is equivalent to $v_l$}, that is when for all $\varepsilon > 0$, there exists an integer $L$ such that for all $l > L$, $|u_l-v_l| < \varepsilon |v_l|$, and $u_l = O(v_l)$ when there exists $c_2 > 0$ such that for all large enough $l$, $u_l < c_2 v_l$.

We will prove some properties of the left-weighting graph by using the notion of adjacency matrix. For more details on graph theory and adjacency matrices, the reader can consult \cite{graphes}. We recall the following definition and proposition, together with the theorem of Perron-Frobenius.
\begin{defi}[adjacency matrix]
Let $G$ be an oriented finite graph, whose vertices are numbered. We call \emph{adjacency matrix of $G$} the matrix whose $(i,j)$-entry contains the number of edges from the vertex $i$ to the vertex $j$.
\end{defi}

\begin{prop}\label{prop:adj}
Let $G$ be an oriented finite graph and $A$ its adjacency matrix. Let $l \in \N$. The $(i,j)$-entry of the matrix $A^l$ contains the number of paths of length $l$ in $G$ linking the vertex $i$ to the vertex $j$.
\end{prop}

\begin{theo*}[Perron-Frobenius]
Let $A$ be a matrix such that there exists $k \in \N^\ast$, such that all entries of $A^k$ are positive. Then the spectral radius of $A$ is positive, is a simple eigenvalue of $A$, and is the unique eigenvalue of maximal module.
\end{theo*}

\begin{lem}\label{lem:connexite}
Each pair of vertices in $\graphe$ is linked by at least one path of length exactly~$5$.
\end{lem}

\begin{proof}
Let us recall that two simple braids $s$ and $t$ are left-weighted if and only if $\init{t} \subset \final{s}$. Let $s_1$ and $s_2$ be two simple braids distinct from $1$ and $\Delta$. There exist $i_1$ and $i_2$ in $\{1, \ldots, n-1\}$ such that $\final{s_1} \supset \{i_1\}$ and $\init{s_2} \subset \{1,\ldots,n-1\}\backslash\{i_2\}$. We will construct some simple braids $x_1,x_2,x_3,x_4$ satisfying:
\begin{itemize}
\item $\init{x_1} = \{i_1\}$,
\item $\final{x_1} = \init{x_2} = \{\lfloor \frac n 2 \rfloor\}$,
\item $\final{x_2} = \init{x_3} = \{1, 3, \ldots, 2\lfloor \frac n 2 \rfloor -1\}$ (the set of odd numbers between $1$ and $n-1$),
\item $\final{x_3} = \init{x_4} = \{1,\ldots,n-1\}\backslash\{\lfloor \frac n 2 \rfloor\}$,
\item $\final{x_4} = \{1,\ldots,n-1\}\backslash\{i_2\}$.
\end{itemize}
Thus, $s_1 \to x_1 \to x_2 \to x_3 \to x_4 \to s_2$ will be a path of length $5$ in the graph $\graphe$.

Here is how we choose the braids $x_1,x_2,x_3,x_4$. We set $x_1 = \sigma_{i_1} \cdots \sigma_{\lfloor \frac n 2 \rfloor}$. The simple braid $x_2$ is the braid corresponding to following permutation:
\[\pi_2 = \begin{pmatrix} 1 & 2 &\cdots & \lfloor \frac n 2 \rfloor & \lfloor \frac n 2 \rfloor+1 & \lfloor \frac n 2 \rfloor+2 & \cdots & n\\
2 & 4 & \cdots & 2\lfloor \frac n 2 \rfloor & 1 & 3 & \cdots &  2 \lceil \frac n 2 \rceil -1\end{pmatrix}\]
As to the braid $x_3$, it is equal to $\bar x_2 \Delta_{1,\lfloor \frac n 2 \rfloor} \Delta_{\lfloor \frac n 2 \rfloor+1,n}$, where $\bar x_2$ is the simple braid of permutation $\pi_2^{-1}$.
Finally, $x_4 = \Delta \sigma_{\lceil \frac n 2 \rceil}^{-1} \cdots \sigma_{i_2}^{-1}$ is the left complement of $\sigma_{i_2} \cdots \sigma_{\lceil \frac n 2 \rceil}$. The braids $x_1$ to $x_4$ are represented for $n=6$ in Figure \ref{fig:chemin}.

\begin{figure}[!h]
\centering
\includegraphics[scale=0.9]{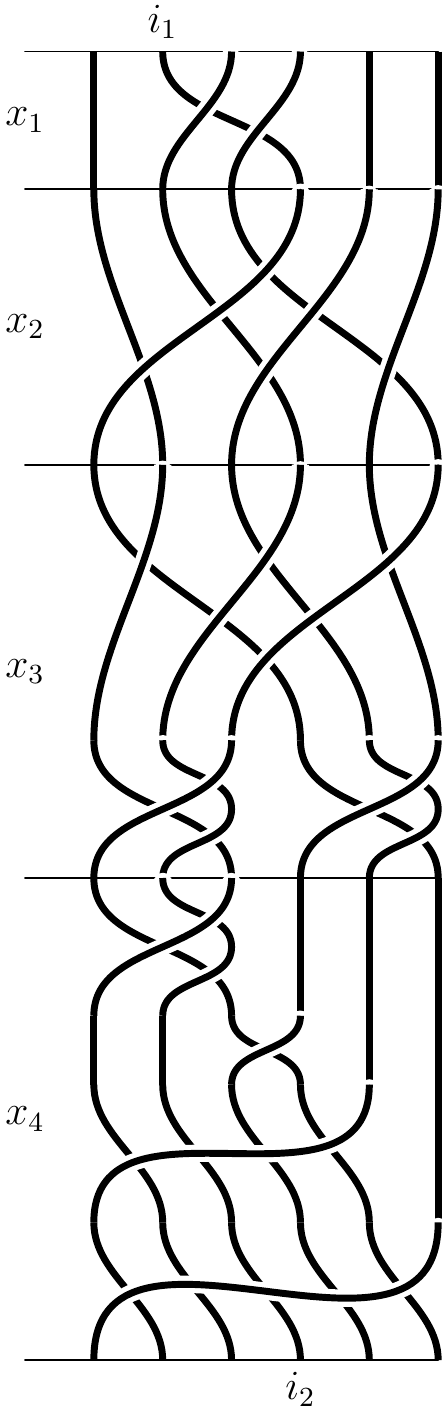}
\caption{Braids $x_1$ to $x_{4}$}
\label{fig:chemin}
\end{figure}

Of course $\init{x_1} = \{i_1\}$ and $\final{x_1} = \{\lfloor \frac n 2 \rfloor\}$. For $x_2$, the permutation $\pi_2$ is increasing on $\{1, \ldots, \lfloor \frac n 2 \rfloor\}$ and on $\{\lfloor \frac n 2 \rfloor+1, \ldots, n\}$, and we have $\pi_2(\lfloor \frac n 2 \rfloor+1) < \pi_2(\lfloor \frac n 2 \rfloor)$, so $\init{x_2} = \{\lfloor \frac n 2 \rfloor\}$. On the other hand, $\pi_2^{-1}(i) > \pi_2^{-1}(i+1)$ if and only if $i$ is odd, hence $\final{x_2} =  \{1, 3, \ldots, 2\lfloor \frac n 2 \rfloor -1\}$. The permutation $\pi_3$ associated with $x_3$ first applies $\pi_2^{-1}$, then reverses the order, on the one hand, of the elements from $1$ to $\lfloor \frac n 2 \rfloor$, and on the other hand, of $\lfloor \frac n 2 \rfloor+1$ to $n$. It follows that $\pi_3(i) > \pi_3(i+1)$ if and only if $i$ is odd, and that $\pi_3^{-1}(i) > \pi_3^{-1}(i+1)$ for all $i$ except $i = \lfloor \frac n 2 \rfloor$. So $\init{x_3} =  \{1, 3, \ldots, 2\lfloor \frac n 2 \rfloor -1\}$ and $\final{x_3} = \{1,\ldots,n-1\}\backslash\{\lfloor \frac n 2 \rfloor\}$. Finally, $x_4$ is the left complement of $\sigma_{i_2} \cdots \sigma_{\lceil \frac n 2 \rceil}$, and thus satisfies $\init{x_4} =  \{1,\ldots,n-1\}\backslash\{n - \lceil \frac n 2 \rceil\} = \{1,\ldots,n-1\}\backslash\{\lfloor \frac n 2 \rfloor\}$ and $\final{x_4} =  \{1,\ldots,n-1\}\backslash\{i_2\}$.
\end{proof}

\begin{lem}\label{lem:graphe}
The following properties are true.
\begin{enumerate}[\rm (i)]
\item There exists a constant $\g$ such that $\nbb(l) \sim \g^{l+1}$.
\item We have $\nb(l) = \Theta(\g^l)$. In particular, for large enough $l$, the proportion $\nbb(l)/\nb(l)$ is bounded below, independently of $l$, by a positive constant.
\item For all path $w$, there exists a constant $\gw < \g$ such that $\nbw(l) = O(\gw^l)$ and $\nbbw(l) = O(\gw^l)$.
\end{enumerate}
\end{lem}

The reader can also consult \cite{Deh}, which contains results and proofs similar to those of this lemma.

\begin{proof}[Proof of Lemma \ref{lem:graphe}]

Let $A$ be the adjacency matrix of the graph $\graphe$. According to Proposition \ref{prop:adj}, $\nbb(l) = \tr(A^{l+1})$ and $\nb(l) = | A^l |_1$, where $| \cdot |_1$ is the sum of all entries in the matrix.

(i) According to Lemma \ref{lem:connexite}, the matrix $A^5$ has positive entries. So we can apply the Perron-Frobenius theorem to $A$, and deduce that $A$ has a unique eigenvalue of maximal module. This value is real and positive, and the associated eigenspace has dimension $1$. We denote by $\g$ this eigenvalue, and by $\lambda_i$ ($i = 1, \ldots, n!-3$) the others (not necessarily distinct and not necessarily real). We have $\tr(A^{l+1}) = \g^{l+1} + \lambda_1^{l+1} + \cdots + \lambda_{n!-3}^{l+1}$, hence $\nbb(l) \sim \g^{l+1}$ when $l$ tends to infinity.

(ii) There exists an invertible matrix $P$ such that $P A P^{-1}$ is in Jordan normal form, and we can calculate
\[|P A^l P^{-1}|_1 = \g^l + \sum p_i(\lambda_i)\]
where the $p_i$ are some polynomials of degree $l$. We deduce again the equivalence $|P A^l P^{-1}|_1 \sim \g^l$. Furthermore, $|PA^lP^{-1}|_1 = \Theta(|A^l|_1)$, so $\nb(l) = \Theta(\g^l)$.

We deduce that $\nbb(l)/\nb(l) = \Theta(1)$, and in particular, that for large enough $l$, this ratio is bounded below, independently of $l$, by a positive constant.

(iii) We construct from $\graphe$ a graph $\graphek$ (where, we recall, $k$ is the length of the path $w$)  as follows: the vertices in $\graphek$ are the paths of length $k-1$ in $\graphe$, and two paths $w_1 = (s_1\to \cdots \to s_k)$ and $w_2 = (t_1 \to \cdots \to t_{k})$ are linked by an edge if and only if $s_2 = t_1, s_3 = t_2, \ldots, s_{k} = t_{k-1}$. Thus, the edges of $\graphek$ correspond to the paths of length $k$ in $\graphe$. We denote by $A_k$ the adjacency matrix of $\graphek$.

If $s$ and $t$ are two vertices in $\graphe$, a path of length $l \geqslant k$ from $s$ to $t$ in $\graphe$ corresponds to a path of length $l-k+1$ from ${(s \to s_2 \to \cdots \to s_{k})}$ to ${(t_1 \to \cdots \to t_{k} \to t)}$ in $\graphek$ for some $s_2, \ldots, s_{k-1},t_1, \ldots, t_{k-2}$. This leads to the following consequences. As each pair of vertices in $\graphe$ is linked by a path of length $5$ (Lemma \ref{lem:connexite}), each pair of vertices in $\graphek$ is linked by a path of length exactly $k+4$. Furthermore, as the number of paths of length $l$ in $\graphe$ is a $\Theta(\g^l)$, it is the same for the number of paths of length $l$ in $\graphek$. As $A_k^{k+4}$ has positive entries, $A_k$ satisfies the hypothesis of the Perron-Frobenius theorem, and we deduce, as in (ii), that the number of paths of length $l$ in $\graphek$ is a $\Theta(\g_{(k)}^l)$ where $\g_{(k)}$ is the spectral radius of $A_k$. The two asymptotic estimates obtained ensure that $\g_{(k)} = \g$.

Moreover, avoiding a path of length $k$ in $\graphe$ is equivalent to avoiding an edge in $\graphek$. Let $\graphekw$ be the graph obtained from $\graphek$ by removing the edge $a_w$ corresponding to $w$. We denote by $\tilde A_k$ its adjacency matrix, and by $\gw$ the spectral radius of this matrix. Then, the number of paths of length $l-k+1$ in $\graphek$ is a $O(\gw^l)$: indeed, as in (ii), there exists an invertible matrix $Q$ such that $|Q \tilde A_k^{l-k+1} Q^{-1}|_1$ is a sum of polynomials of degree $l-k+1$ in the eigenvalues of $\tilde A_k$. As these eigenvalues are, in module, not greater than the spectral radius $\gw$, we deduce that $|Q \tilde A_k^{l-k+1} Q^{-1}|_1 = O(\gw^{l-k+1}) = O(\gw^l)$, and then, that $\nbw(l) = |\tilde A_k^{l-k+1}|_1 = O(\gw^l)$.

As for the number of loops of length $l+1$ with marked base point in $\graphe$, their number is not greater than the number of paths of length $l$, and so we have also $\nbbw(l) = O(\gw^l)$.

It remains to prove that $\gw < \g$.

Given two vertices $w_1 = (s_1 \to \cdots \to s_k)$ and $w_2 = (t_1 \to \cdots \to t_k)$ of $\graphek$, there always exists a path of length $l_0 = 2k+9$ in $\graphek$ from $w_1$ to $w_2$ passing through the edge $a_w$: indeed, it suffices to go with a path of length $k+4$ until the starting vertex of $a_w$, to go through the edge $a_w$, and to go again with a path of length $k+4$ until $w_2$. This means that there are strictly more paths of length $l_0$ from $w_1$ to $w_2$ in $\graphek$ than in $\graphekw$. That is to say, the matrix $A_k^{l_0} - \tilde A_k^{l_0}$ has positive entries.
Let $\varepsilon > 0$ be such that $A_k^{l_0} - (\tilde A_k^{l_0} + \varepsilon I)$ has still positive entries (where $I$ is the identity matrix). The spectral radius of $A_k^{l_0}$ is $\g^{l_0}$, the one of $\tilde A_k^{l_0} + \varepsilon I$ is $\gw^{l_0}+\varepsilon$. Recall that the spectral radius of a matrix $B$ is the limit of $\Vert B^k \Vert^{\frac 1 k}$ when $k$ tends to infinity, where $\Vert \cdot \Vert$ is any matrix norm. By choosing for $\Vert \cdot \Vert$, for example, the infinity-norm, we deduce that, as the entries of $A_k^{l_0}$ are all greater than those of $(\tilde A_k^{l_0} + \varepsilon I)$, we have $\g^{l_0} \geqslant \gw^{l_0}+\varepsilon$, and thus $\g > \gw$.
\end{proof}

\begin{rem}
By similar arguments, we obtain finer results, on the number of paths that do not contain $w$ \emph{in a more localized area of the path}. More precisely, if $\beta$ is a path of length $l$, and if $a_1$, $a_2$, $a_3$ are functions of $l$ taking values in $\N$, with $a_1+a_3$ and $a_2$ nondecreasing functions that tends to infinity when $l$ tends to infinity, and such that $a_1(l) + a_2(l) + a_3(l) = l$, we can cut the path $\beta$ into three path $\beta_1$, $\beta_2$ and $\beta_3$ of respective lengths $a_1(l)$, $a_2(l)$ and $a_3(l)$. The number of paths $\beta$ of length $l$ whose ``middle part'' $\beta_2$ does not contain the path $w$ is a $\Theta(\gw^{a_2(l)} \g^{a_1(l) + a_3(l)}) = \Theta(\gw^{a_2(l)} \g^{l-a_2(l)})$.
\end{rem}

\section{Genericity of pseudo-Anosov braids}

\subsection{Proportion of rigid braids}

\begin{prop}
Let $l \in \N^\ast$. Among the braids $\beta$ such that $\inf \beta = 0$ and $\sup \beta = l$, the proportion of rigid braids is bounded below independently of $l$ by a positive constant.
\end{prop}

\begin{proof}
According to the unicity of the left normal form of a braid, the set of all braids $\beta$ such that $\inf \beta = 0$ and $\sup \beta = l$ is in bijection with the set of paths of length $l$ in the left-weighting graph $\graphe$. The set of \emph{rigid} braids of infimum $0$ and supremum $l$ is in bijection with the set of paths of length $l$, for which there is an edge from the last to the first vertex. Hence, the proposition is a corollary of Lemma  \ref{lem:graphe}, (ii).
\end{proof}

\subsection{Proportion of non pseudo-Anosov braids with infimum 0}

The aim of this section is to show that, among the rigid braids of some fixed infimum and canonical length~$l$, the proportion of non pseudo-Anosov braids tends to $0$ when $l$ tends to infinity. For this, we can use the following theorem, due to Gonz\'alez-Meneses and Wiest \cite{GMW} (Theorem 5.16):

\begin{theo}\label{theo:GMW}
Let $\beta$ be a non-periodic, reducible braid which is rigid. Then there is some positive integer $k \leqslant n$ such that one of the following conditions holds:
\begin{enumerate}[\rm (1)]
\item $\beta^k$ preserves a round curve, or
\item $\inf(\beta^k)$ and $\sup(\beta^k)$ are even, and either $\Delta^{-\inf(\beta^k)}\beta^k$ or $\beta^{-k}\Delta^{\sup(\beta^k)}$ is a positive braid which preserves an almost round curve whose corresponding interior strands do not cross.
\end{enumerate}
\end{theo}

Let us also state the following theorem of Bernadete, Gutierrez and Nitecki (Theorem~5.7 in \cite{BGN}) as given in \cite{Calvez} (Theorem 1):
\begin{prop}\label{prop:BGN}
Let $x \in \B_n$, seen as a mapping class in $\mcg(\disque n, \partial \disque n)$, with left normal form $x = \Delta^p x_1 \cdots x_r$. Let $\mathcal C$ be a round curve in $\disque n$. If $x(\mathcal C)$ is round, then $\Delta^p x_1\cdots x_m(\mathcal C)$ is round for all $m = 1, \ldots, r$.
\end{prop}

\begin{notation}
In what follows we shall use the following two braids, written in normal form as follows: 
$$
\spa=\sigma_1\sigma_3\cdots\sigma_{2\lfloor \frac n 2 \rfloor -1} \ . \  \sigma_1\sigma_3\cdots\sigma_{2\lfloor \frac n 2 \rfloor -1}\sigma_2\sigma_4 \cdots \sigma_{2\lceil \frac n 2 \rceil - 2} \text{ \ \ (length 2)}$$
$$\spb=\Delta_{2,n} \sigma_1 \ . \ \sigma_1 \ . \ \sigma_1\sigma_2\cdots \sigma_{n-1} \ . \ \sigma_{n-1} \text{ \ \ (length 4)}
$$
(see Figures \ref{fig:rondes} and \ref{fig:semirondes}.)
\end{notation}

\begin{figure}[h]
\centering
\includegraphics{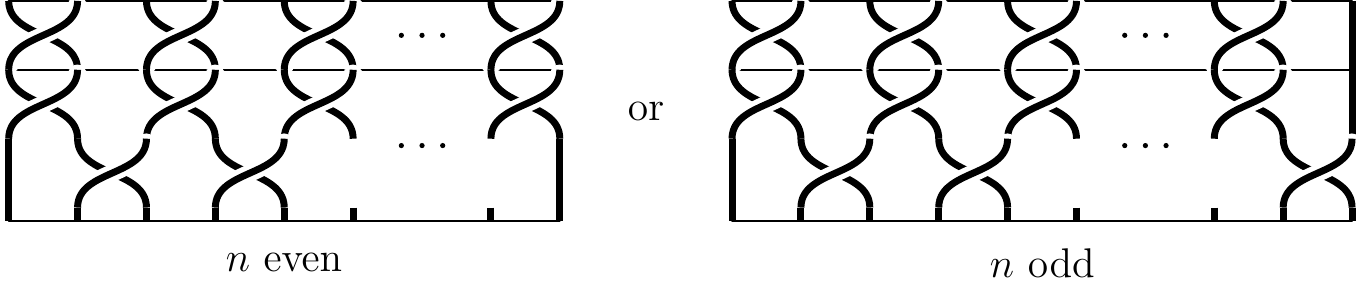}
\caption{A braid sending no round curve to a round curve}
\label{fig:rondes}
\end{figure}

\begin{figure}[h]
\centering
\includegraphics{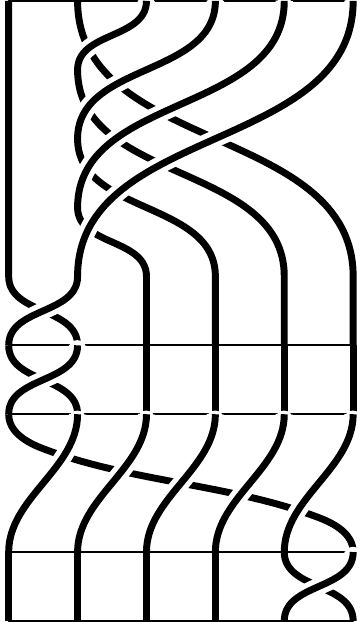}
\caption{A braid where each pair of strands crosses in some factor, and does not cross in some other factor}
\label{fig:semirondes}
\end{figure}

\begin{prop}\label{prop:propred}
A rigid braid whose normal form contains both $\spa$ and $\spb$ as subwords is pseudo-Anosov.
\end{prop}

\begin{proof} Let us study a rigid braid $\beta$, denoting $\inf(\beta) = \epsilon$ and the canonical length of~$\beta$ as $l$.

First, we remark that there is no periodic rigid braid except $\Delta^\epsilon$. Indeed, if a braid~$\beta$ is rigid and has canonical length at least~$1$, then its left normal form is of the shape 
$$\beta=\Delta^\epsilon s_1 s_2 \cdots s_l$$ 
where $(s_i, s_{i+1})$ ($i = 1, \ldots, l-1$) and $(s_l,\tau^{-\epsilon}s_1)$ are left-weighted. Therefore, the normal form of a power of this braid is of the shape 
$$
\beta^k=\Delta^{k\epsilon} s_1^{(1)}s_2^{(1)}\cdots s_l^{(1)}s_1^{(2)}s_2^{(2)}\cdots s_l^{(2)}\cdots\cdots s_1^{(k)}s_2^{(k)}\cdots s_l^{(k)}
$$
where $s_i^{(j)}=\tau^{(k-j)\epsilon}(s_i)$, which is never a power of~$\Delta$ when $l\geqslant 1$.

Let us now deal with the possibility that $\beta$ might be reducible. 
According to Theorem~\ref{theo:GMW}, there are three possible cases.

The first case correspond to the case (1) of the theorem. A power of $\beta$ preserves a round curve. The rigidity of $\beta$ implies that the normal form of a power of $\beta$ contains the normal form of~$\beta$ (except the initial factors~$\Delta$) as a subword. According to Proposition~\ref{prop:BGN}, we deduce that there exists a round curve whose image by $\beta$ is still a round curve.

The second case is the case where a $k$-th power of $\beta$ is such that $\Delta^{-\inf(\beta^k)}\beta^k=\Delta^{-k\epsilon}\beta^k$ preserves an almost round curve whose interior strands do not cross. If the normal form of $\beta$ is $\Delta^\epsilon s_1 s_2 \cdots s_l$, then, as before, $\Delta^{-k\epsilon}\beta^k$ has normal form
$$ 
\Delta^{-k\epsilon} \beta^k=s_1^{(1)}s_2^{(1)}\cdots s_l^{(1)}s_1^{(2)}s_2^{(2)}\cdots s_l^{(2)}\cdots\cdots s_1^{(k)}s_2^{(k)}\cdots s_l^{(k)}
$$
This word has two strands that never cross, and hence so does the word $s_1 s_2\cdots s_l$ representing $\Delta^{-\epsilon}\beta$.

Let us look at the third case. This time, it is the braid  $\beta^{-k}\Delta^{\sup(\beta^k)} = \beta^{-k}\Delta^{k(l+\epsilon)}$ which has two strands that do not cross. 
Note that this braid has infimum~$0$ and supremum~$k\cdot l$. Therefore in the braid 
$$\Delta^{k\cdot l}\cdot\left(\beta^{-k}\Delta^{k(l+\epsilon)}\right)^{-1}=\Delta^{-k\epsilon}\beta^k$$ 
(whose normal form was given in the previous paragraph) there are two strands which cross in every single factor.
Hence the same is true for the word $s_1 s_2\cdots s_l$ representing $\Delta^{-\epsilon}\beta$: it has two strands which cross in every factor.

Now, a braid~$\beta$ whose normal form contains $\spa$ cannot send any round curve to a round curve. The reason for this is that no round curve is sent to a round curve by this sequence of two simple braids (see Figure \ref{fig:rondes}), and according to Proposition~\ref{prop:BGN}, this is also the case for the whole braid~$\beta$. Similarly, a braid containing $\spb$ cannot 
contain two strands that do not cross at all, or that cross in every single factor (see Figure \ref{fig:semirondes}). This completes the proof.
\end{proof}


We now restrict our attention temporarily to the case of braids with infimum $0$.

\begin{lem}\label{lem:nbpAinf0}
The number of braids of infimum $0$ and supremum $l$, which are rigid and pseudo-Anosov, is a $\Theta(\g^l)$, where $\g$ is the constant of Lemma \ref{lem:graphe}.
\end{lem}

\begin{proof}
Let us denote by $\Omega$ the set of rigid braids of infimum $0$ and supremum $l$. We also denote by $E_1 \subset \Omega$ the subset of the braids that do not contain, in their normal form, the normal form of $\spa$ as a subword. We denote by $E_2 \subset \Omega$ the subset of the braids that do not contain, in their normal form, the normal form of $\spb$ as a subword. 

According to Lemma~\ref{lem:graphe}, with the same notations, the cardinality $\card{\Omega}$ is equivalent to $\g^{l+1}$. 

Still from Lemma~\ref{lem:graphe}, we also have estimates $\card{E_1} =  O(\gp^{l})$ where $\gp < \g$ and $\card{E_2} = O(\gpp^{l})$ where $\gpp < \g$. 
Thus the cardinality of the set $E_1 \cup E_2$, whose cardinality is less than $c(\gp^{l}+\gpp^{l})$ for a suitable constant $c > 0$, and this set contains all rigid braids of infimum~$0$ and supremum~$l$ which are non pseudo-Anosov. 



As $\gp < \g$ and $\gpp < \g$, the number of braids of infimum~$0$ and supremum~$l$ which are rigid and pseudo-Anosov, is still of the order of $\g^l$.
\end{proof}

\subsection{Arbitrary infimum}

Let us consider the Cayley graph of the braid group, with generators the simple braids. The following lemma, which is an immediate consequence of Lemma 3.1 in \cite{CM}, gives the possible left normal forms for a braid that is at distance $l$ from the neutral element in this graph.

\begin{lem}
Let $\beta$ be a braid at distance $l$ from the neutral element in the Cayley graph. Then the left normal form of $\beta$ has one of the following shapes:
\begin{enumerate}[{\rm (i)}]
\item $\beta = \Delta^{-l} s_1 \cdots s_k$, $k \in \{0, \ldots, l-1\}$,
\item $\beta = \Delta^{-k} s_1 \cdots s_l$, $k \in \{0, \ldots, l\}$,
\item $\beta = \Delta^k s_1 \cdots s_{l-k}$, $k \in \{1, \ldots, l\}$.
\end{enumerate}
\end{lem}

The following theorem is a generalization of the results previously obtained in the particular case of a zero infimum.

\begin{theo}
For large enough $l$, among all braids at distance $l$ from the neutral element in the Cayley graph, the proportion of rigid pseudo-Anosov braids is bounded below by a positive constant.
\end{theo}

\begin{proof}
First, let us make a remark: a braid $\beta$ is pseudo-Anosov if and only if $\Delta^2 \beta$ is pseudo-Anosov. The same is true when we replace ``pseudo-Anosov'' by ``rigid''. Thus, a braid with left normal form $\Delta^p s_1 \cdots s_r$ with $p$ even is pseudo-Anosov (respectively rigid) if and only if $s_1 \cdots s_r$ is.

According to Lemma \ref{lem:nbpAinf0}, there exists a constant $c_1 > 0$ such that for all large enough~$l$, the number of rigid pseudo-Anosov braids of the form $s_1 \cdots s_l$ is bounded below by $c_1 \g^l$. Consequently, the number of rigid pseudo-Anosov braids of the form $\Delta^{-k} s_1 \cdots s_l$ with $k \in \{0, \ldots, l\}$ and $k$ even is bounded below by $c_1 \frac l 2 \g^l$.

Furthermore, let us bound above the total number of braids at distance $l$ of the neutral element. According to Lemma \ref{lem:graphe}, there exists a constant $c_2$ such that the number of braids with normal form $s_1 \cdots s_k$ is bounded above by $c_2 \g^k$. So:
\begin{enumerate}[{\rm (i)}]
\item the number of braids with normal form $\Delta^{-l} s_1 \cdots s_k$ ($0 \leqslant k < l$) is bounded above by $c_2 (1 + \cdots + \g^{l-1})$,
\item the number of braids with normal form $\Delta^{-k} s_1 \cdots s_l$ ($0 \leqslant k \leqslant l$) is bounded above by $c_2 l \g^l$,
\item the number of braids with normal form $\Delta^k s_1 \cdots s_{l-k}$ ($0 < k \leqslant l$) is bounded above by $c_2 (1 + \cdots + \g^{l-1})$.
\end{enumerate}
As $c_2 (1+ \cdots + \g^{l-1}) \sim \frac{c_2}{\g-1} \g^l$, if we replace $c_2$ by an even larger constant, we can suppose that, in the cases (i) and (iii), the number of braids is bounded above by $\frac{c_2}{\g-1} \g^l$. Finally, the proportion of rigid pseudo-Anosov braids among all braids of length $l$ is bounded below by
\[\frac{c_1 \frac l 2 \g^l}{\frac{c_2}{\g-1}\g^l+c_2 l \g^l + \frac{c_2}{\g-1}\g^l}=\frac{c_1}{2c_2}\cdot\frac{1}{1+\frac 2{l(\g-1)}} \geqslant \frac{c_1}{2c_2}\cdot\frac{1}{1+\frac 2{\g-1}} > 0,\]
which completes the proof.
\end{proof}

\begin{cor}\label{cor:boule}
For large enough $l$, in the $l$-ball of the Cayley graph, the proportion of rigid pseudo-Anosov braids is bounded below independently of $l$ by a positive constant.
\end{cor}

\begin{proof}
The number of braids in the $k$-sphere is of the order of $k\g^k$, and the $l$-ball is the union of the $k$-spheres for $k \leqslant l$. We deduce that the number of braids in the $l$-ball is of the order of $l\g^l$, that is to say, of the order of the number of braids in the $l$-sphere. 
So the proportion of rigid pseudo-Anosov braids remains of the order of a constant.
\end{proof}

\end{document}